\documentclass{amsart}
\usepackage[utf8]{inputenc}
\usepackage[T1]{fontenc}
\usepackage[all]{xy}
\usepackage{lipsum}
\usepackage{url}
\usepackage{tikz}
\usepackage{stackrel}
\usepackage{color}
\def\red{\textcolor{black}}

\usepackage{amsmath,amsthm,amssymb}

\usepackage{xcolor}

\usepackage{graphicx}

\newtheorem{theorem}{Theorem}[section]
\newtheorem{lemma}[theorem]{Lemma}
\newtheorem{example}[theorem]{Example}
\newtheorem{proposition}[theorem]{Proposition}
\theoremstyle{definition}
\newtheorem{definition}[theorem]{Definition}

\newtheorem{remark}[theorem]{Remark}

\numberwithin{equation}{section}

\begin{document}

\vspace{0.5in}

\renewcommand{\bf}{\bfseries}
\renewcommand{\sc}{\scshape}
\vspace{0.5in}

\title[Motion planning algorithms]%
{Collision-free spatial motion of rigid bodies via topological complexity \\ }

\author{Cesar A. Ipanaque Zapata}
\address{Departamento de Matem\'{a}ticas, Universidade de S\~{a}o Paulo, 
Instituto de Ci\^{e}ncias Matem\'{a}ticas e de Computa\c{c}\~{a}o -
USP, Avenida Trabalhador S\~{a}o-carlense, 400 - Centro CEP:
13566-590 - S\~{a}o Carlos - SP, Brasil}
\curraddr{Departamento de Matem\'{a}ticas, Centro de Investigaci\'{o}n y de Estudios Avanzados del I. P. N.
Av. Instituto Polit\'{e}cnico Nacional n\'{u}mero 2508,
San Pedro Zacatenco, Mexico City 07000, M\'{e}xico}
\email{cesarzapata@usp.br}



\thanks {The author wishes to acknowledge support for this research from grant\#2018/23678-6 and grant\#2016/18714-8, S\~ao Paulo Research Foundation (FAPESP). Also, the author is very grateful to Jes\'{u}s Gonz\'{a}lez for their comments and encouraging remarks which were of invaluable mental support.}

\begin{abstract} The Topological complexity a la Farber $\text{TC}(-)$ is a homotopy invariant  which have interesting applications in Robotics, specifically, in the robot motion planning problem. In this work we calculate the topological complexity of the configuration space of $k$ distinct rigid bodies without collisions in $\mathbb{R}^d$, for $d=2,3$. Furthermore, we present optimal algorithms which can be used in designing practical systems controlling motion of many rigid bodies moving in space without collisions. The motion planning algorithms we present in this work are easily implementable in practice.
\end{abstract}

\maketitle


\section{\bf Introduction}

Consider a multi-robot system consisting of $k$ mobile robots $R_1,\ldots,R_k$, which are rigid bodies and we consider them as compact subsets of $\mathbb{R}^d$ ($d\geq 2$), moving in Euclidean space $\mathbb{R}^d$ without collisions. We will suppose that the diameters of all robots are equal to $r>0$, i.e., $\text{diam}(R_i)=r>0,$ for any $i=1,\ldots,k$. The origin $"0"$ and coordinate basis vectors $\{e_1,e_2,\ldots,e_d\}$ of $\mathbb{R}^d$  will be referred to as \textit{the world frame} (we will also say the \textit{reference frame}). Recall that $e_i=(0,\ldots,0,1,0\ldots,0)\in \mathbb{R}^d$. We can always attach a coordinate frame (local frame) to a rigid object under consideration. In this work we shall deal with the pose (state or configuration) and the displacement of rectangular frames (see \cite{bajdRobotics}). We describe a rigid body by its orientation of the object and its position (e.g. the position of its center of mass), see Figure \ref{robot}.

Therefore, orientation-position determines the pose of a rigid object. The orientation of the local frame of the object and the position of the object are respect to the world frame.
\begin{figure}[h]
 \centering
\begin{tikzpicture}[x=.6cm,y=.6cm]
\draw[->](-2,0)--(-1,0); \draw[->](-2,0)--(-1,0.5); \draw[->](-2,0)--(-2,1);  
\node [anchor=west] at (-1,0) {\tiny$e_1$};
\node [above] at (-2,1) {\tiny$e_3$};
\node [above] at (-1,1) {\tiny$e_2$};
\node [below] at (-2,0) {\tiny$0$};
\fill[black!30] (2.5,1.5) ellipse (1 and 0.5); 
\draw[black, thick](3.3,1.3)--(3.3, 1.7); \draw(2.1,1.8)--(3.3, 1.7);\draw(2.1,1.2)--(3.3, 1.3); \filldraw[black] (2.3,1.8) circle (2pt); \filldraw[black] (3.1,1.7) circle (2pt); \filldraw[black] (2.3,1.2) circle (2pt); \filldraw[black] (3.1,1.3) circle (2pt);
\draw[black, thick](1.9,1.1)--(1.9, 1.9); \draw(2.1,1.1)--(2.1, 1.9); \draw(1.9,1.1)--(2.1, 1.1); \draw(1.9,1.9)--(2.1, 1.9); 
\filldraw[black] (2.5,1.5) circle (1pt); \node[anchor=east] at (2.5,2.5) {\tiny$(1)$};
\draw[dashed](2.5,1.5)--(4.5,1.5); \draw[dashed](2.5,1.5)--(2.5,3.5); \draw[dashed](2.5,1.5)--(3.5,2);
\filldraw[black] (3.3,1.5) circle (1pt); \node [below] at (3.5,1.5) {\tiny$\theta$};
\draw[->](2.5,1.5)--(3.5,1.5);
\node [below] at (2.6,1.6) {\tiny$p$};
\node [below] at (4.5,1.5) {\tiny$u_1$};
\node [above] at (2.5,3.5) {\tiny$u_3$};
\node [above] at (3.5,2) {\tiny$u_2$};
\draw[dashed](-2,0)--(2.5,1.5);
\fill[black!30] (7.5,-2.5) ellipse (1 and 0.5); 
\draw[black, thick](6.8,-2.2)--(7.2,-2); \draw(6.8,-2.2)--(7.1,-2.7); \draw(7.2,-2)--(8.2,-2.4); 
\draw[black, thick](7.2,-3)--(8.5,-2.5); \draw(7.2,-3)--(7,-2.8); \draw(8.5,-2.5)--(8.3,-2.3); \draw(7,-2.8)--(8.3,-2.3); 
\filldraw[black] (7,-2.6) circle (2pt); \filldraw[black] (6.8,-2.4) circle (1.5pt);
\filldraw[black] (8,-2.3) circle (2pt); \filldraw[black] (7.4,-2.1) circle (1.5pt);
\filldraw[black] (7.5,-2.5) circle (1pt); \node[below] at (7.5,-3.5) {\tiny$(1)$};
\draw[dashed](7.5,-2.5)--(6,-1.1); \draw[dashed](7.5,-2.5)--(5.5,-2.7); \draw[dashed](7.5,-2.5)--(7.5,-0.5);
\filldraw[black] (7,-2.1) circle (1pt); \node [above] at (7,-2.1) {\tiny$\theta^\prime$};
\draw[->](7.5,-2.5)--(7,-2.1);
\node [below] at (7.5,-2.6) {\tiny$p^\prime$};

\node [below] at (6,-1.1) {\tiny$u_1^\prime$};
\node [above] at (7.5,-0.5) {\tiny$u_3^\prime$};
\node [above] at (5.5,-2.7) {\tiny$u_2^\prime$};
\draw[dashed](-2,0)--(7.5,-2.5);
\end{tikzpicture}
\caption{The Robot $(1)$ has initial state $(\theta,p)=(id,p)$ and final state $(\theta^\prime,p^\prime)$.}
 \label{robot}
\end{figure}
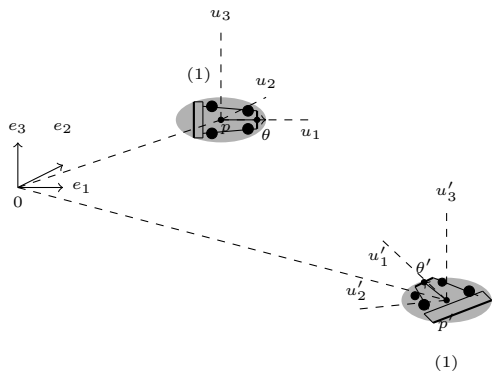

Recall that in general the \textit{configuration space} or \textit{state space} of a system $\mathcal{S}$ is defined as the space of all possible states of $\mathcal{S}$. The configuration space to the multi-robot system is the product $\left(SO(d)\right)^k\times F_r(\mathbb{R}^d,k)$, \[\{(\theta_1,\ldots,\theta_k;p_1,\ldots,p_k)\mid ~~(\theta_1,\ldots,\theta_k)\in \left(SO(d)\right)^k \text{ and } (p_1,\ldots,p_k)\in F_r(\mathbb{R}^d,k) \}, \]
where $F_r(\mathbb{R}^d,k)=\{(p_1,\ldots,p_k)\in (\mathbb{R}^d)^k\mid~~ \parallel p_i-p_j\parallel> 2r \text{ for } i\neq j\}$ is the configuration space of all possible arrangements of $k$ nonoverlapping disks of radius $r$ in $\mathbb{R}^d$, equipped with subspace topology of the Cartesian power $(\mathbb{R}^d)^k$.

Note that the $i-th$ coordinates $(\theta_i,p_i)$ of a point $(\theta_1,\ldots,\theta_i,\ldots,\theta_k;p_1,\ldots,p_i,\ldots,p_k)\in \left(SO(d)\right)^k\times F_r(\mathbb{R}^d,k)$ represents the orientation-position of the $i-th$ moving object, so that the condition $\parallel p_i-p_j\parallel> 2r$ reflects the collision-free requirement. 

In robotics, we need to know the configuration space $C$, the workspace $W$ and thus the work map $f:C\to W$ (see \cite{bajdRobotics}). The \textit{robot workspace} consists of all points that can be reached by the robot
end-point, i.e., the space of all tasks. The workspace $W$ is often described as a subspace of some Euclidean space $\mathbb{R}^N$. The \textit{work map} or \textit{kinematic map} is a continuous map from the configuration $C$ to the workspace $W$, that is, it is a continuous map \[f:C\to W\] which assigns to each state of the configuration space the position of the end-effector at that state.  This map is an important object to be considered when implementing algorithms controlling the task performed by the robot manipulator.

A more common task for mobile robots is to request them to navigate in an indoor environment, as shown in Figure \ref{robot}.
In this work the task of each robot consists of the point that can be reached by the pose of the robot, that is, a robot might be asked to perform tasks such as arriving at a particular place with a particular orientation. Thus, the workspace of this $k$ robots coincides with the configuration space $\left(SO(d)\right)^k\times F_r(\mathbb{R}^d,k)$ and the work map is the identity map. 

\begin{remark}
We will study in another paper when the task of each robot consists of the point that can be reached by the position of the robot, that is, a robot might be asked to perform tasks such as arriving at a particular place but the particular orientation is not specified in the task. Thus, the workspace coincides with $F_r(\mathbb{R}^d,k)$ and the work map coincides with the projection $\left(SO(d)\right)^k\times F_r(\mathbb{R}^d,k)\to F_r(\mathbb{R}^d,k)$.
\end{remark}

Our work is considered as an instance of planning. \textit{The collision-free robot motion planning problem} consists in controlling simultaneously the motion of these $k$ robots without collisions, where one is interested in initial-final states of the robots. To solve this problem we need to find a \emph{collision-free optimal motion planning algorithm} on state space $\left(SO(d)\right)^k\times F_r(\mathbb{R}^d,k)$ (see Section \ref{sec2}). The challenges of modern robotics (see, for example Latombe \cite{latombe2012robot} and LaValle \cite{lavalle2006planning}) is design explicit and suitably optimal motion planners. To give collision-free optimal algorithms we need to know the smallest possible number of regions of continuity for any collision-free motion planning algorithm, that is, (see Section \ref{sec2}) the value of the \textit{topological complexity} a la Farber $\text{TC}(\left(SO(d)\right)^k\times F_r(\mathbb{R}^d,k))$.

In this paper we compute the value of $\text{TC}(\left(SO(d)\right)^k\times F_r(\mathbb{R}^d,k))$ for $d=2,3$. 
\begin{theorem}[Principal theorem]\label{principal-theorem}
Let $k\geq 2$, we have
\begin{enumerate}
    \item $\text{TC}((\mathbb{S}^1)^k\times F_r(\mathbb{R}^2,k))=3k-2$.
    \item $\text{TC}((\mathbb{RP}^3)^k\times F_r(\mathbb{R}^3,k))=5k-1$.
\end{enumerate}
\end{theorem}
Furthermore, we present optimal tame motion planning algorithms (see Definition \ref{tame}) in $\left(SO(d)\right)^k\times F_r(\mathbb{R}^d,k)$ with $3k-2$ (for $d=2$) and $5k-1$ (for $d=3$) regions of continuity, respectively. These algorithms work for any $k\geq 2$ and they are easily implementable in practice.


\section{Preliminary results}\label{sec2}

The notion of topological complexity was introduced by Farber, which is defined in terms of motion planning algorithms for a robot moving between initial-final configurations~\cite{farber2003topological}.

\medskip
For a topological space $X$, let $PX$ denote the space of paths $\gamma:[0,1]\to X$, equipped with the compact-open topology. One has the evaluation fibration \begin{equation}\label{evaluation-fibration}
    e:PX\to X\times X,~e(\gamma)=\left(\gamma(0),\gamma(1)\right).
\end{equation} 
A \textit{motion  planning  algorithm} is  a  section $s\colon X\times X\to PX$ of  the  fibration  $e$, i.e.~ a (not necessarily continuous) map satisfying $e\circ s=id_{X\times X}$. A (global) continuous motion planning algorithm in $X$ exists if and only if the space $X$ is contractible~\cite{farber2003topological}. This fact gives, in a natural way, the definition of the following numerical invariant. The \textit{topological complexity} $\text{TC}(X)$ of a path-connected space $X$ is the Schwarz genus of the evaluation fibration~(\ref{evaluation-fibration}). In  other  words the topological complexity of $X$ is the smallest positive integer $\text{TC}(X)=n$ for which  the product $X\times X$ is covered by $n$ open subsets $X\times X=U_1\cup\cdots\cup U_n$ such that for any $i=1,2,\ldots,n$ there exists a continuous section $s_i:U_i\to PX$ of $e$ 
over $U_i$ (i.e., $e\circ s_i=id$. Here $id$ denote the inclusion map $U_i\hookrightarrow X\times X$). Any motion  planning  algorithm $s:=\{s_i:U_i\to PX\}_{i=1}^{n}$ is called \textit{optimal} if $n=\text{TC}(X)$.

\medskip One of the basic properties of TC is its homotopy invariance, that is, if $X$ and $Y$ are homotopy equivalent then $\text{TC}(X)=\text{TC}(Y)$. Furthermore, their motion  planning  algorithms are explicitly related.

\begin{remark}\cite{farber2003topological} [Homotopy invariance]\label{homotopy-invariance}
Suppose $X$ \textit{dominates} $Y$, i.e., there are maps $f:X\to Y$ and $g:Y\to X$ such that $f\circ g\simeq id_Y$. Choose a homotopy $H:Y\times [0,1]\to Y$ with $H_0=id_Y$ and $H_1=f\circ g$. Let $\text{TC}(X)=n$ and let $s:=\{s_i:U_i\to PX\}_{i=1}^{n}$ be a motion  planning  algorithm to $X$ with $X\times X=U_1\cup\cdots\cup U_n$ and $e\circ s_i=id$. Set $V_i:=(g\times g)^{-1}(U_i)\subseteq Y\times Y$ for $i=1,\ldots,n$. Define $\hat{s}_i:V_i\to PY$ by the formula
\begin{equation}\label{up-homotopy-mp}
  \hat{s}_i(y_1,y_2)(t)= \begin{cases}
    H_{3t}(y_1), & \hbox{$0\leq t\leq \frac{1}{3}$;} \\
    f\left(s_i(g(y_1),g(y_2))(3t-1)\right), & \hbox{$\frac{1}{3}\leq t\leq \frac{2}{3}$;}\\
    H_{3-3t}(y_2), & \hbox{$\frac{2}{3}\leq t\leq 1$.}
\end{cases}.
\end{equation} One has $Y\times Y=V_1\cup\cdots\cup V_n$ is an open covering and $e\circ \hat{s}_i=id$. Thus $$\hat{s}:=\{\hat{s}_i:V_i\to PY\}_{i=1}^{n}$$ is a motion planning algorithm to $Y$ and hence $\text{TC}(Y)\leq n=\text{TC}(X)$.

 In particular, if $X$ and $Y$ are homotopy equivalent we have $\text{TC}(X)=\text{TC}(Y)=n$. Furthermore, if $s:=\{s_i:U_i\to PX\}_{i=1}^{n}$ is  an optimal motion  planning  algorithm to $X$ then $\hat{s}:=\{\hat{s}_i:V_i\to PY\}_{i=1}^{n}$, as above, is an optimal motion  planning  algorithm to $Y$.
\end{remark}

Let $\mathbb{K}$ be a field. The singular cohomology $H^{*}(X;\mathbb{K}):=H^{*}(X)$ is a graded $\mathbb{K}-$algebra with multiplication \[\cup:H^{*}(X)\otimes H^{*}(X)\longrightarrow H^{*}(X)\] given by the cup-product. The tensor product $H^{*}(X)\otimes H^{*}(X)$ is also a graded $\mathbb{K}-$algebra with the multiplication
\[(u_1\otimes v_1)\cdot (u_2\otimes v_2):=(-1)^{\deg(v_1)\deg(u_2)}u_1u_2\otimes v_1v_2,\]
 where $\deg(v_1)$ and $\deg(u_2)$ denote the degrees of cohomology classes $v_1$ and $u_2$ respectively. The cup-product $\cup$ is a homomorphism of $\mathbb{K}-$algebras. 
 
According to \cite{farber2003topological} the kernel of homomorphism $\cup$ is \textit{the ideal of the zero-divisors} of $H^{*}(X)$. The \textit{zero-divisors-cup-length} of $H^{*}(X)$  (with coefficients in $\mathbb{K}$), denoted $zcl_{\mathbb{K}}(H^{*}(X))$, is the length of the longest nontrivial product in the ideal of the zero-divisors of $H^{*}(X)$.

The following Proposition \ref{prop-1} gives the general properties of topological complexity.

\begin{proposition}\label{prop-1} 
\begin{enumerate}
\item  (\emph{\cite{farber2003topological}, Theorem 7}) Let $\mathbb{K}$ be a field and $X$ be a path-connected topological space. We have \[1+zcl_{\mathbb{K}}(H^{*}(X))\leq \text{TC}(X).\] 

\item (\emph{\cite{cohen2011topological}, Lemma 2.1}) Suppose that $X$ and $Y$ are path connected finite CW complexes and let $\mathbb{K}$ be a field. Then
\[zcl_{\mathbb{K}}\left(H^{*}(X)\otimes H^{*}(X)\right)\geq zcl_{\mathbb{K}}(H^{*}(X))+zcl_{\mathbb{K}}(H^{*}(Y)).\]

\item (\emph{\cite{farber2003topological}, Theorem 11}) Let $X$ and $Y$ be any path-connected metric spaces, then\[\text{TC}(X\times Y)\leq \text{TC}(X)+\text{TC}(Y)-1.\]
\end{enumerate}
\end{proposition}

\subsection{Tame motion planning algorithm}
\medskip
Despite the definition of $\text{TC}(X)$ deals with open subsets of $X\times X$ admitting continuous sections of the evaluation fibration (\ref{evaluation-fibration}), for practical purposes, the construction of explicit motion planning algorithms is usually done by partitioning the whole space $X\times X$ into pieces, over each of which~(\ref{evaluation-fibration}) has a continuous section.  Since any such partition necessarily contains subsets which are not open (recall $X$ has been assumed to be path-connected), we need to be able to operate with subsets of $X\times X$ of a more general nature.

Recall that, a topological space $X$ is an \textit{Euclidean Neighbourhood Retract} (ENR) if it can be embedded into an Euclidean space $\mathbb{R}^d$ with an open neighbourhood $U$, $X\subset U\subset \mathbb{R}^d$, admitting a retraction $r:U\to X,$ $r\mid_U=id_X$.

\begin{example}
A subspace $X\subset \mathbb{R}^d$ is an ENR if and only if it is locally compact and locally contractible, see~\red{\cite[Chap.~4, Sect.~8]{dold2012lectures}}. This implies that all finite-dimensional polyhedra, smooth manifolds and semi-algebraic sets are ENRs.
\end{example}

\begin{definition}\label{tame}
Let $X$ be an ENR. A motion planning algorithm $s:X\times X\to PX$ is said to be \textit{tame} if $X\times X$ splits as a pairwise disjoint union $X\times X=F_1\sqcup\cdots\sqcup F_n$, where each $F_i$ is an ENR, and each restriction $s\mid_{F_i}:F_i\to PX$ is continuous. The subsets $F_i$ in such a decomposition are called \emph{domains of continuity} for $s$.
\end{definition}

\begin{proposition}\emph{(\cite[Proposition 2.2]{rudyak2010higher})}\label{rudi}
For an ENR $X$, $\text{TC}(X)$ is the minimal number of domains of continuity $F_1,\ldots,F_n$ for tame motion planning algorithms $s:X\times X\to PX$.
\end{proposition}

\begin{remark}[Tame motion planner in a product]\label{product-mp}
In general, to get a motion planning algorithm in the product $X\times Y$ requires partitions of unity subordinate to covers from motion planning algorithms to $X$ and $Y$, respectively (\emph{\cite{farber2003topological}, Theorem 11}). However, we will recall here (see \cite{farber2004instabilities}, Section 12) a simple explicit construction of a tame motion planning algorithm in $X\times Y$ with $\text{TC}(X)+\text{TC}(Y)-1$ domains of continuity, under an additional assumption. This of course suits best our implementation-oriented objectives.

Let $s:=\{s_i:F_i\to PX\}_{i=1}^{n}$ be an optimal tame motion planner in $X$ and let $\sigma:=\{\sigma_j:G_j\to PY\}_{j=1}^{m}$ be an optimal tame motion planner in $Y$. Assume that the motion planner $s$, satisfies the following condition:
\begin{center}
   '\textit{Topologically disjoint condition}'- the closure of each set $F_i$ is contained in the union $F_1\cup\cdots\cup F_i$, in other words, it require that all sets of the form $F_1\cup\cdots\cup F_i$ be closed.
\end{center}
 Similarly, we will assume that $\sigma$ is a tame motion planner in $Y$ such that all sets of the form $G_1\cup\cdots\cup G_j$ are closed. Then we will set \begin{equation}\label{union}
    W_\ell= \bigsqcup_{i+j=\ell} F_i\times G_j, ~~\ell=2,\ldots,n+m.
\end{equation}
The sets are ENRs and form a partition of $(X\times X)\times (Y\times Y)=(X\times Y)\times (X\times Y)$. Our
assumptions guarantee that each product $F_i\times G_j$ is closed in $W_\ell$, where $l=i+j$. Since
different products in the union \ref{union} are disjoint, we see that the maps $s_i\times\sigma_j$, where
$i+j=\ell$, determine a continuous motion planning strategy over each set $W_\ell$. Furthermore, we note that the motion planner in $X\times Y$ as above also satisfies the 'Topologically disjoint condition', i.e.,  all sets of the form $W_2\cup\cdots\cup W_\ell$ be closed.
\end{remark}

\section{Poof of Theorem \ref{principal-theorem}}\label{proof-principalTH}
The proof of Theorem \ref{principal-theorem} is accomplished proving the next three lemmas.

\begin{lemma}[TC for products]\label{prop2}
Let $\mathbb{K}$ be a field and $X$ and $Y$ be any path-connected finite CW complexes. If $\text{TC}(X)=zcl_{\mathbb{K}}(X)+1$ and $\text{TC}(Y)=zcl_{\mathbb{K}}(Y)+1$, then
\[\text{TC}(X\times Y)=\text{TC}(X)+\text{TC}(Y)-1.\] Furthermore, $\text{TC}(X\times Y)=zcl_{\mathbb{K}}(X\times Y)+1.$

In particular, for any $k\geq 2$, $\text{TC}(\underbrace{X\times\cdots\times X}_{k \text{ times }})=k\text{TC}(X)-(k-1)$.
\end{lemma}
\begin{proof}
It follows from Proposition \ref{prop-1}.
\end{proof}

By \cite{farber2003topological}, we have \[\text{TC}(\mathbb{S}^n)=zcl_{\mathbb{Z}_2}(\mathbb{S}^n)+1=\left\{
  \begin{array}{ll}
    2, & \hbox{for $n$ odd;} \\
    3, & \hbox{for $n$ even.}
  \end{array}
\right.\] Moreover, it is easy to see (or see \cite{farber2003topologicalproject}) $\text{TC}(\mathbb{RP}^3)=zcl_{\mathbb{Z}_2}(\mathbb{RP}^3)+1=4.$ Hence, we have the following statement.

\noindent\begin{lemma}\label{exem}  For any $k\geq 2$, one has
\begin{enumerate}
    \item $TC(\underbrace{\mathbb{S}^1\times\cdots\times\mathbb{S}^1}_{k \text{ times }})=zcl_{\mathbb{Z}_2}(\underbrace{\mathbb{S}^1\times\cdots\times\mathbb{S}^1}_{k \text{ times }})+1=k+1.$
    \item $TC(\underbrace{\mathbb{RP}^3\times\cdots\times\mathbb{RP}^3}_{k \text{ times }})=zcl_{\mathbb{Z}_2}
    (\underbrace{\mathbb{RP}^3\times\cdots\times\mathbb{RP}^3}_{k \text{ times }})+1=3k+1.$
\end{enumerate}
\end{lemma}

\noindent Next, we will study the homotopy type of $F_r(\mathbb{R}^d,k)$. Recall that $$F_r(\mathbb{R}^d,k)=\{(p_1,\ldots,p_k)\in (\mathbb{R}^d)^k\mid~~ \parallel p_i-p_j\parallel> 2r \text{ for } i\neq j\}$$ denote \textit{the ordered configuration space of all possible arrangements of $k$ nonoverlapping disks of radius $r$} in $\mathbb{R}^d$, equipped with subspace topology of the Cartesian power $(\mathbb{R}^d)^k$.

\medskip Let $\chi:F(\mathbb{R}^d,k)\longrightarrow \mathbb{R}$ given by \[\chi(p):=\dfrac{1}{2}\min\{\parallel p_i-p_j\parallel: \text{ for } i\neq j\} \text{ with } p=(p_1,\ldots,p_k)\in F(\mathbb{R}^d,k),\] where $\parallel \cdot\parallel$ is the Euclidean norm and $F(\mathbb{R}^d,k)=\{(p_1,\ldots,p_k)\in (\mathbb{R}^d)^k\mid~~p_i\neq p_j \text{ for } i\neq j\}$ is the classical ordered configuration space \cite{fadell1962configuration}. Note that $F_r(\mathbb{R}^d,k)=\chi^{-1}(r,+\infty)$ and $F(\mathbb{R}^d,k)=\chi^{-1}(0,+\infty).$

\begin{lemma}[Homotopy type of $F_r(\mathbb{R}^d,k)$]\label{hard-spheres}
For any $r>0$ and $k\geq 2$, one has $F_r(\mathbb{R}^d,k)$ and $F(\mathbb{R}^d,k)$ are homotopy equivalent.
\end{lemma}
\begin{proof}

Define \[\rho:F(\mathbb{R}^d,k)\to F_r(\mathbb{R}^d,k),~~\rho(p)=\left(\dfrac{\chi(p)+2r}{\chi(p)}\right)p,\]
note that $\rho$ is well-defined, because $\chi(p)>0$ and $\chi\left(\dfrac{\chi(p)+2r}{\chi(p)}p\right)=\chi(p)+2r> 2r> r$. Moreover, $\rho$ is continuous, because $\chi$ is. 

Let $i:F_r(\mathbb{R}^d,k)\hookrightarrow F(\mathbb{R}^d,k)$ denote the inclusion map. We will see that $i\circ\rho\simeq id_{F(\mathbb{R}^d,k)}$ and $\rho\circ i\simeq id_{F_r(\mathbb{R}^d,k)}$. Define \[H:F(\mathbb{R}^d,k)\times [0,1]\to F(\mathbb{R}^d,k),~~(p,t)\mapsto H(p,t):=\dfrac{\chi(p)+2rt}{\chi(p)}p.\] One has $H$ is well-defined and continuous. Furthermore, $H(p,0)=p$ and $H(p,1)=(i\circ \rho)(p)$ for $p\in F(\mathbb{R}^d,k)$. Therefore, $H$ is a homotopy between $id_{F(\mathbb{R}^d,k)}$ and $i\circ \rho$.

Similarly, consider the restriction map \[\hat{H}:=H\mid_{F_r(\mathbb{R}^d,k)\times [0,1]}:F_r(\mathbb{R}^d,k)\times [0,1]\to F_r(\mathbb{R}^d,k),~~(p,t)\mapsto \hat{H}(p,t):=\dfrac{\chi(p)+2rt}{\chi(p)}p.\]
Note that $\hat{H}$ is well-defined, because
\begin{eqnarray*}
\chi(\hat{H}(p,t)) &=& \chi(H(p,t))\\
&=& \chi\left(\dfrac{\chi(p)+2rt}{\chi(p)}p\right)\\
&=& \chi(p)+2rt\\
&\geq& \chi(p)\\
&>& r, \text{ for any } (p,t)\in F_r(\mathbb{R}^d,k)\times [0,1].
\end{eqnarray*} Hence, $\hat{H}$ is a homotopy between $id_{F_r(\mathbb{R}^d,k)}$ and $\rho\circ i$. 
\end{proof}

\noindent \textit{Proof of Theorem \ref{principal-theorem}} We recall that $\text{TC}$ is a homotopy invariant, so by Lemma \ref{hard-spheres}, it is sufficient to calculate the topological complexity $\text{TC}((SO(d))^k\times F(\mathbb{R}^d,k))$.  By \cite{farber2004topological}, we have \[\text{TC}(F(\mathbb{R}^d,k))=zcl_{\mathbb{Z}_2}(F(\mathbb{R}^d,k))+1=\left\{
  \begin{array}{ll}
    2k-2, & \hbox{if $d=2$;} \\
    2k-1, & \hbox{if $d=3$.}
  \end{array}
\right.\] Then by Lemmas \ref{prop2} and \ref{exem} we obtain our theorem. $\Box$

\section{Motion planning algorithms}

In this section, we present optimal tame motion planning algorithms in:
\begin{enumerate}
    \item[(1)] product of odd-dimensional spheres,
    \item[(2)] product of 3-dimensional real projective spaces,
    \item[(3)] the configuration space $F_r(\mathbb{R}^d,k)$. Here, the algorithms will be induce from the algorithms given by the authors in \cite{zapata2019multitasking},
    \item[(4)] the product $(\mathbb{S}^1)^k\times F_r(\mathbb{R}^2,k)$ and $(\mathbb{RP}^3)^k\times F_r(\mathbb{R}^3,k)$.
\end{enumerate} All the algorithms are easily implementable in practice.

\subsection{On product of odd-dimensional spheres}

Assume that $m$ is odd. We recall that the topological complexity $\text{TC}(S^m)=2$ and  $\text{TC}(\underbrace{S^m\times\cdots\times S^m}_{k \text{ times }})=k+1$ for any $k\geq 2$ (see Lemma \ref{exem}). In this section, using Remark \ref{product-mp}, we will give an optimal tame motion planning algorithm on $\underbrace{S^m\times\cdots\times S^m}_{k \text{ times }}$ having $k+1$ domains of continuity $W_k\ldots,W_{2k}$ such that each $W_\ell$ satisfies the 'Topological disjoint condition' (see Remark \ref{product-mp}), i.e.,  $\overline{W_\ell}\subset \bigcup_{i\leq \ell} W_i$.

Let $v$ denote a fixed unitary tangent vector field on $S^{m}$, say $v(x_1,y_1,\ldots,x_\ell,y_\ell)=(-y_1,x_1,\ldots,-y_\ell,x_\ell)$ with $m+1=2\ell$. A tame motion planning algorithm to $S^m$ is given by $s:=\{s_i:U_i\to PS^m\}_{i=1}^{2}$, where \begin{eqnarray*}
F_1&=& \{(\theta_1,\theta_2)\in S^m\times S^m\mid ~~\theta_1= -\theta_2\},\\
F_2&=&  \{(\theta_1,\theta_2)\in S^m\times S^m\mid ~~\theta_1\neq -\theta_2\},
\end{eqnarray*} for all  $(\theta_1,\theta_2)\in F_1$,
$$
s_1(\theta_1,\theta_2)(t) =  \begin{cases} 
\dfrac{(1-2t)\theta_1+2tv(\theta_1)}{\parallel (1-2t)\theta_1+2tv(\theta_1) \parallel}, &\hbox{ $0\leq t\leq\dfrac{1}{2}$;}\\
\dfrac{(2-2t)v(\theta_1)+(2t-1)\theta_2}{\parallel (2-2t)v(\theta_1)+(2t-1)\theta_2 \parallel}, &\hbox{ $\dfrac{1}{2}\leq t\leq 1$,}
\end{cases} $$
 and $$ 
s_2(\theta_1,\theta_2)(t) = \dfrac{(1-t)\theta_1+t\theta_2}{\parallel (1-t)\theta_1+t\theta_2 \parallel} \text{ for all } (\theta_1,\theta_2)\in F_2.
$$ We note that \begin{equation}\label{inter}
    F_1\cap F_2=\emptyset, \overline{F_1}=F_1 \text{ and } \overline{F_2}=S^m.
\end{equation} Let $k\geq 2$ and for each $\ell=k,\ldots,2k$ define $$ W_{\ell}=\bigsqcup_{i_1+\cdots+i_k=l}F_{i_1}\times\cdots\times F_{i_k}. $$  One has that each $W_{\ell}$ is an ENR and $W_{k},\ldots,W_{2k}$ form a partition of $\left(\underbrace{S^m\times\cdots\times S^m}_{k \text{ times }}\right)\times \left(\underbrace{S^m\times\cdots\times S^m}_{k \text{ times }}\right)$. In view of (\ref{inter}), the sets assembling each $W_\ell$ are \textit{topologically disjoint}, because $\left(\overline{U_{i_1}\times\cdots\times U_{i_k}}\right)\cap \left( U_{i_1^\prime}\times\cdots\times U_{i_k^\prime}\right)=\varnothing$ for $(i_1,\ldots.i_k)\neq(i_1^\prime,\ldots,i_k^\prime)$ and $i_1+\cdots+i_k=l=i_1^\prime+\cdots+i_k^\prime$. Hence, the sets $W_\ell$ are ENR's covering $\left(\underbrace{S^m\times\cdots\times S^m}_{k \text{ times }}\right)\times \left(\underbrace{S^m\times\cdots\times S^m}_{k \text{ times }}\right)$ on each of which the corresponding algorithms $s_1$ and $s_2$ assemble a continuous motion planning algorithm. We have thus constructed a tame motion planning algorithm, say $s$, in $\underbrace{S^m\times\cdots\times S^m}_{k \text{ times }}$ having $k+1$ regions of continuity $W_k,\ldots,W_{2k}$. Furthermore, each $W_\ell$ satisfies $\overline{W_\ell}\subset \bigsqcup_{i\leq \ell}W_i$.

\subsection{On product of 3-dimensional projective spaces}
We recall that the topological complexity $\text{TC}(\mathbb{R}P^3)=4$ and for any $k\geq 2$,  $\text{TC}(\underbrace{\mathbb{R}P^3\times\cdots\times\mathbb{R}P^3}_{k \text{ times }})=3k+1$ (see Lemma \ref{exem}). In this section, again using Remark \ref{product-mp}, we will give an optimal tame motion planning algorithm on $\underbrace{\mathbb{R}P^3\times\cdots\times\mathbb{R}P^3}_{k \text{ times }}$ having $3k+1$ domains of continuity $X_k,\ldots,X_{4k}$ such that each $X_\ell$ satisfies the \textit{'Topological disjoint condition'}, i.e., $\overline{X_\ell}\subset \bigcup_{j\leq \ell} X_j$. 

For our purposes, using the idea from \cite{farber2004instabilities}, we give an optimal tame motion planning algorithm on $\mathbb{R}P^3$ having $4$ domains of continuity $E_1,E_2,E_3,E_4$ such that each $E_i$  satisfies the 'Topological disjoint condition' (see Remark \ref{product-mp}). Then we will apply Remark \ref{product-mp}. 

Here we consider the real projective space $\mathbb{R}P^d=\dfrac{S^d}{x\sim -x}$ as the quotient space from $S^d$ under the antipodal action. Consider the open covering \[U_1\cup\cdots\cup U_{d
+1}=\mathbb{R}P^d,\] where for each $i=1,\ldots,d+1$, $U_i=\{[x_1,\ldots,x_{d+1}]\in \mathbb{R}P^d:~~x_i\neq 0\}$. 
For each $i=1,\ldots,d+1$ define a map $\varphi_i:U_i\to \mathbb{R}^d$ by $$
\varphi_i[x_1,\ldots,x_{d+1}]=\left(\dfrac{x_1}{x_{i}},\ldots,\dfrac{x_{i-1}}{x_{i}},\dfrac{x_{i+1}}{x_{i}},\ldots,\dfrac{x_{d+1}}{x_{i}}\right)
$$ One has $\varphi_i$ is a homeomorphism, because it has a continuous inverse given by $$\psi_i(x_1,\ldots,x_d)=\left[\dfrac{1}{\left(x_1^2+\cdots+x_d^2+1\right)^{1/2}}(x_1,\ldots,x_{i-1},1,x_{i},\ldots,x_{d})\right].$$ Consider the linear homotopy $H:\mathbb{R}^d\times [0,1]\to \mathbb{R}^d$ given by $$H(x,t)=(1-t)x.$$ Now, for each $i=1,\ldots,d+1$, $U_i$ is contractible. In fact, we can define the homotopy $H^i:U_i\times [0,1]\to \mathbb{R}P^d$ by $$
H^i([x_1,\ldots,x_{d+1}],t)=\psi_i\left(H\left(\varphi_i[x_1,\ldots,x_{d+1}],t\right)\right).
$$

On the other hand, for each $i=1,\ldots,d+1$, set \[f_i:\mathbb{R}P^d\to[0,1],~~f_i([x_1,\ldots,x_{k+1}])=x_i^2.\] On has $f_i$ are well-defined smooth functions. The support\footnote{Recall that the support $supp(f)$ of a continuous function $f:X\to \mathbb{R}$ is defined as the closure of the set $\{x\in X:~~f(x)\neq 0\}$.} of $f_i$ being the closure of $U_i$. Indeed the set $\{[x_1,\ldots,x_{k+1}]\in \mathbb{R}P^d:~~f_i([x_1,\ldots,x_{k+1}])\neq 0\}$ is the subset $U_i$. Moreover, for any $[x]\in \mathbb{R}P^d$, $$f_1[x]+\cdots+f_{d+1}[x]=1.$$

Let a subset $V_i\subset\mathbb{R}P^d$, where $i=1,\ldots,d+1$, be defined by the following system of inequalities
$$
\begin{cases}
f_j[x]<\dfrac{2j}{(d+1)(d+2)},& \hbox{ for all $j<i$,}\\
f_i[x]\geq \dfrac{2i}{(d+1)(d+2)}.
\end{cases}
$$ Note that each $\dfrac{i}{(d+1)(d+2)}$ is a regular value of the function $f_i$, so each $V_i$ is a manifold with boundary and hence an ENR. Furthermore, one easily checks that:
\begin{itemize}
    \item $V_i$ is contained in $U_i$; therefore, the homotopy $H^i:U_i\times [0,1]\to \mathbb{R}P^d$ restricts onto $V_i$ and defines a homotopy $H^i$ over $V_i$;
    \item the sets $V_i$ are pairwise disjoint, $V_i\cap V_j=\varnothing$ for $i\neq j$;
    \item $V_1\cup\cdots\cup V_{d+1}=\mathbb{R}P^d$.
    \item each $V_i$ satisfies the 'Topological disjoint condition', i.e., $\overline{V_i}\subset \bigcup_{j\leq i}V_j$.
\end{itemize}

Now, recall that $\mathbb{R}P^3$ is a Lie group under the quaternionic product \begin{eqnarray*}
[x_1,x_2,x_3,x_{4}]\cdot [y_i,y_2,y_3,y_{4}] &=& [\langle x,(y_1,-y_2,-y_3,-y_4)\rangle, \langle x,(y_2,y_1,y_4,-y_3)\rangle,\\ & & \langle   x,(y_3,-y_4,y_1,y_2)\rangle, \langle x,(y_4,y_3,-y_2,y_1)\rangle],
\end{eqnarray*} with unit $[1,0,0,0]$ and inverse (given by the quaternionic conjugation) $[x_1,x_2,x_3,x_{4}]^{-1}=[x_1,-x_2,-x_3,-x_{4}]$.

For $i=1,2,3,4$ set $$E_i=\{([x],[y])\in\mathbb{R}P^3\times\mathbb{R}P^3:~~[x][y]^{-1}\in V_i\}.$$ It is clear that $E_1\cup E_2\cup E_3\cup E_4=\mathbb{R}P^3\times\mathbb{R}P^3$, the sets $E_i$ are pairwise disjoint, each $E_i$ is an ENR and each $E_i$ satisfies the 'Topological disjoint condition'. Then we may define $\sigma_i:E_i\to P\left(\mathbb{R}P^3\right)$ by the formula \begin{equation}
   \sigma_i([x],[y]) = H^i([x][y]^{-1},t)\cdot [y].
\end{equation} It is a continuous motion planning over $E_i$. Hence, $\sigma=\{s_i:E_i\to P\left(\mathbb{R}P^3\right)\}_{i=1}^{4}$ is an optimal tame motion planner on $\mathbb{R}P^3$ and each $E_i$ satisfies $\overline{E_i}\subset \bigcup_{j\leq i} E_i$. 

 In view of Remark \ref{product-mp}, let $k\geq 2$ and for each $\ell=k,\ldots,4k$ define $$ X_{\ell}=\bigsqcup_{i_1+\cdots+i_k=l}E_{i_1}\times\cdots\times E_{i_k}. $$  One has that each $X_{\ell}$ is an ENR and $X_{k},\ldots,X_{4k}$ form a partition of $\left(\underbrace{\mathbb{R}P^3\times\cdots\times \mathbb{R}P^3}_{k \text{ times }}\right)\times \left(\underbrace{\mathbb{R}P^3\times\cdots\times \mathbb{R}P^3}_{k \text{ times }}\right)$. The sets assembling each $X_\ell$ are \textit{topologically disjoint}. Hence, the sets $X_\ell$ are ENR's covering $\left(\underbrace{\mathbb{R}P^3\times\cdots\times \mathbb{R}P^3}_{k \text{ times }}\right)\times \left(\underbrace{\mathbb{R}P^3\times\cdots\times \mathbb{R}P^3}_{k \text{ times }}\right)$ on each of which the corresponding algorithms $\sigma_i$ assemble a continuous motion planning algorithm. We have thus constructed a tame motion planning algorithm (say $\sigma$) in $\underbrace{\mathbb{R}P^3\times\cdots\times \mathbb{R}P^3}_{k \text{ times }}$ having $3k+1$ regions of continuity $X_k,\ldots,X_{4k}$. Furthermore, each $X_{\ell}$ satisfies  $\overline{X_\ell}\subset \bigcup_{i\leq \ell}X_i$.

\subsection{On the configuration space $F_r(\mathbb{R}^d,k)$}

In this section we present a tame motion planning algorithm on $F_r(\mathbb{R}^d,k)$ having $2k-1$ domains of continuity. The algorithm works for any $r>0$, $d\geq 2$ and $k\geq 2$; this algorithm is optimal when $d$ is odd.

For our purposes, we may recall the algorithm on the configuration space $F(\mathbb{R}^d,k)$ given by the authors in \cite{zapata2019multitasking}, for any $d\geq 2$, say $\omega=\{\omega_\ell:Y_\ell\to PF(\mathbb{R}^d,k)\}_{\ell=2}^{2k}$. Denote by  $p:\mathbb{R}^d\to \mathbb{R},(x_1,\ldots,x_q)\mapsto x_1$ the projection in the first coordinate. For  a  configuration $C\in F(\mathbb{R}^d,k)$,  where $C=(x_1,\ldots,x_{k})$ with $x_i\in\mathbb{R}^d, ~x_i\neq x_j$ for $i\neq j$, consider the set of projection points \[P(C) =\{p(x_1),\ldots,p(x_{k})\},\] $p(x_i)\in \mathbb{R}$, $i=1,\ldots,k$. The cardinality of this set will be denoted $\text{cp}(C)$. Note that $\text{cp}(C)$ can be any number $1, 2,\ldots,k$. The algorithm $\omega=\{\omega_\ell:Y_\ell\to PF(\mathbb{R}^d,k)\}_{\ell=2}^{2k}$ has domains of continuity $Y_2,Y_3,\ldots, Y_{2k}$, where \begin{equation*}
    Y_\ell=\bigcup_{i+j=\ell}A_i\times A_j,
\end{equation*} and $A_i$ is the set of all configurations $C\in F(\mathbb{R}^d,k)$, with $\text{cp}(C)=i$.

We now go back to the notation introduced in Section \ref{proof-principalTH} where we constructed a homotopy equivalence $\rho:F(\mathbb{R}^d,k)\to F_r(\mathbb{R}^d,k),~~\rho(p)=\left(\dfrac{\chi(p)+2r}{\chi(p)}\right)p,$ whose homotopy inverse is the inclusion map $i:F_r(\mathbb{R}^d,k)\to F(\mathbb{R}^d,k)$. Together with the  homotopy $\hat{H}:F_r(\mathbb{R}^d,k)\times [0,1]\to F_r(\mathbb{R}^d,k),~~(p,t)\mapsto \hat{H}(p,t):=\dfrac{\chi(p)+2rt}{\chi(p)}p.$ The map $\hat{H}$ is a homotopy between $id_{F_r(\mathbb{R}^d,k)}$ and $\rho\circ i$. By Remark \ref{homotopy-invariance}, the optimal tame motion planning algorithm $\omega=\{\omega_\ell:Y_\ell\to PF(\mathbb{R}^d,k)\}_{\ell=2}^{2k}$ in  $F(\mathbb{R}^d,k)$ induces an optimal tame motion planning algorithm in $F_r(\mathbb{R}^d,k)$, say $\hat{\omega}=\{\hat{\omega}_\ell:Z_\ell\to PF_r(\mathbb{R}^d,k)\}_{\ell=2}^{2k}$, where each $Z_\ell$ is given by $$Z_\ell=\left(i\times i\right)^{-1}(Y_\ell)$$ and each local motion planner $\hat{\omega}_\ell$ by $$
\hat{\omega}_\ell(p,q) = \begin{cases}
    \hat{H}_{3t}(p), & \hbox{$0\leq t\leq \frac{1}{3}$;} \\
    \rho\left(\omega_\ell(p,q)(3t-1)\right), & \hbox{$\frac{1}{3}\leq t\leq \frac{2}{3}$;}\\
    \hat{H}_{3-3t}(q), & \hbox{$\frac{2}{3}\leq t\leq 1$.}
\end{cases}
$$

\subsection{On the configuration space $F_r(\mathbb{R}^d,k)$ for $d$ even}

By \cite{zapata2019multitasking}, we can improve the motion planning algorithm in $F_r(\mathbb{R}^d,k)$ of the previous section under the assumption that $d\geq 2$ is even. The improved motion planning algorithm will have $2k-2$ domains of continuity; this algorithm is optimal.

For our purposes, we may recall the algorithm on the configuration space $F(\mathbb{R}^d,k)$ given by the authors in \cite{zapata2019multitasking}, for any $d\geq 2$ even, say $\Omega=\{\Omega_\ell:M_\ell\to PF(\mathbb{R}^d,k)\}_{\ell=3}^{2k}$. For a configuration $C=(x_1,\ldots,x_k)\in F(\mathbb{R}^d,k)$, consider the affine line $L_C$ through the points $x_1$ and $x_2$, oriented in the direction of the unit vector \[e_C=\dfrac{x_2-x_1}{\mid x_2-x_1\mid},\] and let $L^\prime_C$ denote the line passing through the origin and parallel to $L_C$ (with the same orientation as $L_C$). Let $p_C:\mathbb{R}^d\to L_C$ be the orthogonal projection, and let $\overline{\text{cp}}(C)$ be the cardinality of the set $\{p_C(x_1),\ldots,p_C(x_k)\}$. Note that $\overline{\text{cp}}(C)$ ranges in $\{2,\ldots, k\}$. For $i\in\{2,\ldots,k\}$, let $G_i$ denote the set of all configurations $C\in F(\mathbb{R}^d,k)$ with $\overline{\text{cp}}(C)=i$. The algorithm $\Omega=\{\Omega_\ell:M_\ell\to PF(\mathbb{R}^d,k)\}_{\ell=3}^{2k}$ has domains of continuity $M_3,\ldots, M_{2k}$, where \begin{equation*}
    M_\ell=\bigcup_{i+j=\ell}A_{ij}\cup \bigcup_{r+s=\ell+1}B_{rs},
\end{equation*} and \begin{align*}
\red{A_{ij}:=\{(C,C^\prime)\in G_i\times G_j\colon e_C\neq -e_{C^\prime}\}}\\
\red{B_{ij}:=\{(C,C^\prime)\in G_i\times G_j\colon e_C =  -e_{C^\prime}\}}
\end{align*}

Again, we go back to the notation introduced in Section \ref{proof-principalTH} where we constructed a homotopy equivalence $\rho:F(\mathbb{R}^d,k)\to F_r(\mathbb{R}^d,k),~~\rho(p)=\left(\dfrac{\chi(p)+2r}{\chi(p)}\right)p,$ whose homotopy inverse is the inclusion map $i:F_r(\mathbb{R}^d,k)\to F(\mathbb{R}^d,k)$. Together with the  homotopy $\hat{H}:F_r(\mathbb{R}^d,k)\times [0,1]\to F_r(\mathbb{R}^d,k),~~(p,t)\mapsto \hat{H}(p,t):=\dfrac{\chi(p)+2rt}{\chi(p)}p.$ The map $\hat{H}$ is a homotopy between $id_{F_r(\mathbb{R}^d,k)}$ and $\rho\circ i$. By Remark \ref{homotopy-invariance}, the optimal tame motion planning algorithm $\Omega=\{\Omega_\ell:M_\ell\to PF(\mathbb{R}^d,k)\}_{\ell=3}^{2k}$ in  $F(\mathbb{R}^d,k)$ (for $d$ even) induces an optimal tame motion planning algorithm in $F_r(\mathbb{R}^d,k)$ (for $d$ even), say $\hat{\Omega}=\{\hat{\Omega}_\ell:N_\ell\to PF_r(\mathbb{R}^d,k)\}_{\ell=3}^{2k}$, where each $N_\ell$ is given by $$N_\ell=\left(i\times i\right)^{-1}(M_\ell)$$ and each local motion planner $\hat{\Omega}_\ell$ by $$
\hat{\Omega}_\ell(p,q) = \begin{cases}
    \hat{H}_{3t}(p), & \hbox{$0\leq t\leq \frac{1}{3}$;} \\
    \rho\left(\Omega_\ell(p,q)(3t-1)\right), & \hbox{$\frac{1}{3}\leq t\leq \frac{2}{3}$;}\\
    \hat{H}_{3-3t}(q), & \hbox{$\frac{2}{3}\leq t\leq 1$.}
\end{cases}
$$

\subsection{On the product $(\mathbb{S}^1)^k\times F_r(\mathbb{R}^2,k)$ and $(\mathbb{RP}^3)^k\times F_r(\mathbb{R}^3,k)$}

The optimal tame motion planning algorithms in $(\mathbb{S}^1)^k\times F_r(\mathbb{R}^2,k)$ and $(\mathbb{RP}^3)^k\times F_r(\mathbb{R}^3,k)$ are given, one more time, by the construction given by Remark \ref{product-mp} assembling the algorithms above.

\begin{remark}
We note that the results and motion planning algorithms described in this paper can also be extended to the case of higher topological complexity (in the sense of Rudyak \cite{rudyak2010higher}) and obtain multitasking collision-free optimal motion planning algorithms for rigid bodies (in a similar way as \cite{zapata2019multitasking}).
\end{remark}

\bibliographystyle{plain}

\end{document}